\documentclass[12pt]{amsart}
\usepackage[english]{babel}
\usepackage{amsmath}
\usepackage[usenames]{color}
\usepackage{amssymb}
\usepackage{xcolor}
\usepackage{xfrac}
\usepackage{esint}
\usepackage{anysize}
\usepackage{mathtools}

\newtheorem{theorem}{Theorem}[section]
\newtheorem{corollary}[theorem]{Corollary}
\newtheorem{lemma}[theorem]{Lemma}

\newtheorem{example}[theorem]{Example}

\begin{document}
\title{Some unit square integrals}
\author{Juan Carlos Sampedro}
\address{Department of Mathematics, University of Basque Country, Barrio de Sarriena s/n,48940 Leioa, Spain}
\email{juancarlos131a@gmail.com}

\begin{abstract} 
In this article we prove some identities which allow us to evaluate some multiple unit square integrals. In our examples we will give the value of some double and triple integrals. Then, we prove several classical integral formulas with the help of these identities and we present others that seems to be new. Finally we get double integrals for classical constants and  different expression  for two Ramanujan's integral formulas.
\end{abstract}


\maketitle

\section{Principal theorems}
In this section we will explain the results that we will use to prove the consequences and examples below.
Since we will be dealing with continuous functions defined over intervals of the form $[0,\alpha]$ we state the following consequence 
 of the Riemann integral definition, for this case. 
 \begin{lemma}\label{t}
Let $f(x)$ be a continuous function on $[0,\alpha]$ for $\alpha \in \mathbb{N}$, then
\begin{equation*}
\int _{0}^{\alpha}f\left( x\right) dx=\lim \limits_{n\to \infty }\frac{1}{n } \sum \limits^{n \alpha}_{k=1}f \left( \frac{k}{n} \right)
\end{equation*}
\end{lemma}


We need  also a version for continuous functions on $[0,\infty)$.   We assume also that they are integrable.
\begin{lemma}\label{formula}
Let $f(x)$ be a monotone, continuous and integrable function on $[0,\infty)$, then
\begin{equation}
\int ^{\infty }_{0}f\left( x\right) dx = \lim \limits_{t\to 0}t\sum \limits^{\infty }_{n=1}f\left( nt\right),
\end{equation}
\end{lemma}
\begin{proof}
According to the previous lemma, for $\alpha \in \mathbb{N}$
\begin{equation*}
\int ^{\alpha }_{0}f(x)dx=\lim \limits_{n\to \infty }\frac{1}{n} \sum \limits^{n\alpha }_{k=1}f\left( \frac{k}{n} \right) \end{equation*}
If we consider $g_{n}(\alpha)=\frac{1}{n} \sum \limits^{n\alpha }_{k=1}f\left( \frac{k}{n} \right)$ and we take the limit
\begin{equation*}
\lim \limits_{\alpha \to \infty }\int ^{\alpha }_{0}f(x)dx=\lim \limits_{\alpha \to \infty }\left( \lim \limits_{n\to \infty }\frac{1}{n} \sum \limits^{n\alpha }_{k=1}f\left( \frac{n}{k} \right) \right)=\lim \limits_{\alpha \to \infty} \left( \lim \limits_{n\to \infty } g_{n}(\alpha)\right)
\end{equation*}
Finally, note that $g_{n}(\alpha)$ meet the requirements for the interchange of the limits due to the monotony. This concludes the proof.
\end{proof}


\begin{theorem}\label{theor1}$($Main Theorem$)$
Let $a_n, \, b_n$ and $c_n$ be real numbers with $a_n\neq0$, and $k\geq 2$. If the single integral is convergent then, 
\begin{align*}
& \int ^{1}_{0} \overset{k}{\dots} \int ^{1}_{0}\frac{e^{i\log(\prod \limits^{k}_{n=1}x^{c_{n}}_{n})}\prod \limits^{k}_{n=1}x^{b_{n}-1}_{n}}{-\log(\prod \limits^{k}_{n=1}x^{a_{n}}_{n})} \prod \limits^{k}_{n=1}dx_{n}
\\ =&\int ^{\infty }_{0} \overset{k}{\dots} \int ^{\infty }_{0}\frac{e^{-\sum \limits^{k}_{n=1}b_{n}x_{n}-i(\sum \limits^{k}_{n=1}c_{n}x_{n})}}{\sum \limits^{k}_{n=1}a_{n}x_{n}} \prod \limits^{k}_{n=1}dx_{n}
\\  =&\int ^{\infty }_{0}\frac{1}{\prod \limits^{k}_{n=1}(a_{n}x + b_{n} + c_{n}i)}dx. 
\end{align*}
where $i$ is the imaginary unit.
\end{theorem}


\begin{proof}
From Lemma \ref{formula} we get
\begin{equation*}
\int ^{\infty }_{0}\frac{1}{\prod \limits^{k}_{v=1}(a_{v}x+b_{v}+c_{v}i)}dx=\lim \limits_{t\to 0}t\sum \limits^{\infty }_{n=1}\frac{1}{\prod \limits^{k}_{v=1}(a_{v}nt+b_{v}+c_{v}i)},
\end{equation*}
We generalize the sum defining the next function whith $x_{1}, x_{2},\cdots , x_{k}\in\left[0,1\right]$,
\begin{equation*}
G (x_{1}, x_{2},\cdots , x_{k})=
\sum \limits^{\infty }_{n=1}\frac{\prod \limits^{k}_{v=1}x^{a_{v}nt+b_{v}+c_{v}i}_{v}}{\prod \limits^{k}_{v=1}(a_{v}nt+b_{v}+c_{v}i)},
\end{equation*}
Considering  that $\frac{x^{n}}{n} =\int_{0}^{x} y^{n-1}dy$,
\begin{equation*}
G (x_{1}, x_{2},\cdots , x_{k})= \sum \limits^{\infty }_{n=1}\int ^{x_{k}}_{0} \overset{k}{\dots} \int ^{x_{1}}_{0}\prod \limits^{k}_{v=1}y^{a_{v}nt+b_{v}+c_{v}i-1}_{v}\prod \limits^{k}_{v=1}dy_{v}.\end{equation*}
Now we  change the sum and the integral, and then we sum the geometric series. In this way  we obtain:
\begin{equation*}
\int ^{x_{k}}_{0} \overset{k}{\dots} \int ^{x_{1}}_{0} \prod \limits^{k}_{v=1}y^{b_{v}+c_{v}i-1}_{v}\frac{(\prod \limits^{k}_{v=1}y^{a_{v}}_{v})^{t}}{1-(\prod \limits^{k}_{v=1}y^{a_{v}}_{v})^{t}} \prod \limits^{k}_{v=1}dy_{v}.
\end{equation*}
Note that
\begin{equation*}
G (1,1,\cdots ,1)=\sum \limits^{\infty }_{n=1}\frac{1}{\prod \limits^{k}_{v=1}(a_{v}nt+b_{v}+c_{v}i)},
\end{equation*}
so,
\begin{equation*}
\int ^{\infty }_{0}\frac{1}{\prod \limits^{k}_{v=1}(a_{v}x+b_{v}+c_{v}i)}dx =\lim \limits_{t\to 0}t\int ^{1}_{0} \overset{k}{\dots} \int ^{1}_{0}\prod \limits^{k}_{v=1}x^{b_{v}+c_{v}i-1}_{v}\frac{(\prod \limits^{k}_{v=1}x^{a_{v}}_{v})^{t}}{1-(\prod \limits^{k}_{v=1}x^{a_{v}}_{v})^{t}} \prod \limits^{k}_{v=1}dx_{v}.
\end{equation*}
Finally, we can commute the limit with the integrals and we calculate  the limit
\begin{equation*}
\lim \limits_{t\to 0}\frac{t\left( \prod \limits^{k}_{v=1}x_{v}\right) ^{tc_{i}}}{1-\left( \prod \limits^{k}_{v=1}x_{v}\right) ^{tc_{i}}} =\frac{-1}{\log\left( \prod \limits^{k}_{v=1}x^{c_{i}}_{v}\right) },
\end{equation*}
and finally we get,
\begin{equation*}
\int ^{\infty }_{0}\frac{1}{\prod \limits^{k}_{v=1}(a_{v}x+b_{v}+c_{v}i)} dx=\int ^{1}_{0}\overset{k}{\dots} \int ^{1}_{0}\prod \limits^{k}_{v=1}x^{b_{v}+c_{v}i-1}_{v}\frac{1}{-\log(\prod \limits^{k}_{v=1}x^{a_{v}}_{v})} \prod \limits^{k}_{v=1}dx_{v}.
\end{equation*}
This concludes the proof. For the second formula just make the change of variables  $(\log(x_1),\log(x_2),\ldots) \to (x_1,x_2,\ldots)$.
\end{proof}


The next results are generalizations of the theorem, so we do not include many details of the proofs.

\begin{theorem}\label{theor2}
Let $a_n, \, b_n$ and $c_n$ be real numbers with $a_n\neq0$, $j,m,k\in\mathbb{N}$ and $k\geq2$. If the multiple integral is convergent,
\begin{align*}
&\int ^{1}_{0} \overset{k}{\dots} \int ^{1}_{0}\frac{\left( -1\right) ^{j}\prod \limits^{k}_{n=1}x^{a_{n}+b_{n}i-1}_{n}}{\prod \limits^{j}_{i=1}\log\left( \prod \limits^{k}_{n=1}x^{c_{i}}_{n}\right) } \prod \limits^{k}_{n=1}dx_{n} \\
&=\int ^{\infty}_{0} \overset{k}{\dots} \int ^{\infty}_{0}\frac{e^{-\sum \limits^{k}_{n=1}a_{n}x_{n}-i(\sum \limits^{k}_{n=1}b_{n}x_{n})}}{\prod\limits^{j}_{i=1}\sum \limits^{k}_{n=1}c_{i}x_{n}} \prod \limits^{k}_{n=1}dx_{n} \\
&=\int ^{\infty}_{0} \overset{j}{\dots} \int ^{\infty}_{0}\frac{1}{\prod \limits^{k}_{n=1}\left(a_{n}+b_{n}i+\sum \limits^{j}_{i=1}c_{i}x_{i}\right)} \prod \limits^{j}_{n=1}dx_{n}.
\end{align*}
\end{theorem}

\begin{proof}
The proof is similar to the previous case   if we use the generalized form for $k$ variables in the  integral
\begin{align*}
 &\int ^{\infty }_{0} \overset{j}{\dots} \int ^{\infty }_{0}\frac{1}{\prod \limits^{k}_{n=1}\left(a_{n}+b_{n}i+\sum \limits^{j}_{i=1}c_{i}x_{i}\right)} \prod \limits^{j}_{n=1}dx_{n}\\
 =&\lim \limits_{t\to 0}t^{j}\sum \limits_{n_{1}\geq 1} \overset{j}{\dots} \sum \limits_{n_{j}\geq 1}\frac{1}{\prod \limits^{k}_{n=1}\left(a_{n}+b_{n}i+\sum \limits^{j}_{i=1}c_{i}n_{i}t\right)}, 
\end{align*}
\end{proof}


\begin{theorem}\label{theor99}
Let $a_n, \, b_n$ and $c_n$ be real numbers with $a_n, \neq0$,  $j,m,k\in\mathbb{N}$ and $k\geq2$. If the multiple integral is convergent, then we have
\begin{align*}
&\int ^{1}_{0} \overset{k}{\dots} \int ^{1}_{0}\frac{(-1)^{j}\prod \limits^{k}_{n=1}x^{a_{n}+b_{n}i-1}_{n}}{\prod \limits^{j}_{i=1}\left(\log\left( \prod \limits^{k}_{n=1}x^{c_{i}}_{n}\right)+m \right) } \prod \limits^{k}_{n=1}dx_{n} \\
&=\int ^{\infty}_{0} \overset{k}{\dots} \int ^{\infty}_{0}\frac{e^{-\sum \limits^{k}_{n=1}a_{n}x_{n}-i(\sum \limits^{k}_{n=1}b_{n}x_{n})}}{\prod\limits^{j}_{i=1}\left(\sum \limits^{k}_{n=1}c_{i}x_{n}+m\right)} \prod \limits^{k}_{n=1}dx_{n} \\
&=\int ^{\infty}_{0} \overset{j}{\dots} \int ^{\infty }_{0}\frac{e^{-m \left( \sum \limits^{j}_{i=1}x_{i}\right) }}{\prod \limits^{k}_{n=1}\left(a_{n}+b_{n}i+\sum \limits^{j}_{i=1}c_{i}x_{i}\right)} \prod \limits^{j}_{n=1}dx_{n}
\end{align*}
\end{theorem}
\begin{corollary}\label{theor3}
Let $k$ be a natural number, then,
\begin{equation*}
\sum \limits_{n}\int ^{\infty }_{0}\overset{k}{\dots} \int ^{\infty }_{0}\frac{e^{-n \sum \limits^{k}_{i=1}x_{i}}}{1+\sum \limits^{k}_{i=1}x_{i}} \prod \limits^{k}_{i=1}dx_{i}=\sum \limits_{n }\int ^{1}_{0}\frac{(-1)^{k}}{\left( \log\left( x\right) -n \right) ^{k}} dx.
\end{equation*}
\end{corollary}
\begin{proof}
Just use the last theorem with $j=1$,$a_{n}=1$,$b_{n}=0$ and $c_{i}=1$.
\end{proof}


\section{Aplications}
\begin{theorem}\label{cor0}
Let $a$, $b$, $c$, $d$, $h$, $g$ be real numbers, if the series converges we have
\begin{equation*}
-\int ^{1}_{0}\int ^{1}_{0}\frac{x^{b-1}y^{g-1}}{(1-rx^{c}y^{h})\log(x^{a}y^{d})} dxdy=
\sum \limits^{\infty }_{n=0}r^{n}\frac{\log\left( \frac{d}{a} \right) -\log\left( \frac{hn+g}{cn+b} \right) }{-a(hn+g)+bd+cdn}.
\end{equation*}
\begin{equation*}
\int ^{1}_{0}\int ^{1}_{0}\frac{x^{b-1}y^{h-1}\left( x^{a}y^{d}-1\right) }{\left( 1-r x^{c}y^{e}\right) \log\left( x^{a}y^{d}\right) } dxdy=\sum \limits^{\infty }_{n=0}r^{n}\frac{\log\left( \frac{\left( b+cn\right) \left( d+e+hn\right) }{\left( a+b+cn\right) \left( e+hn\right) } \right) }{bd-ae+cdn-ahn},
\end{equation*}
\end{theorem}
\begin{proof}
\begin{align*}
\sum \limits^{\infty }_{n=0}r^{n}\frac{\log\left( \frac{d}{a} \right) -\log\left( \frac{hn+g}{cn+b} \right) }{-a(hn+g)+bd+cdn} =\sum \limits^{\infty }_{n=0}r^{n}\int ^{\infty }_{0}\frac{1}{(ax+b+cn)(dx+g+hn)} dx,
\end{align*}
We interchange the sum and the integral and then we use the identities of the theorem \ref{theor1}.
\begin{align*}
\sum \limits^{\infty }_{n=0}r^{n}\int ^{1}_{0}\int ^{1}_{0}\frac{x^{b+cn-1}y^{g+hn-1}}{-\log (x^{a}y^{d})} dxdy &=\int ^{1}_{0}\int ^{1}_{0}\sum \limits^{\infty }_{n=0}(rx^{c}y^{h})^{n}\frac{x^{b-1}y^{g-1}}{-\log (x^{a}y^{d})} dxdy \\ &= -\int ^{1}_{0}\int ^{1}_{0}\frac{x^{b-1}y^{g-1}}{(1-rx^{c}y^{h}) \log (x^{a}y^{d})} dxdy.
\end{align*}
The proof of the other formula is similar.
\end{proof}

\begin{corollary}\label{cor1}
Let $a$, $b$, $c$, $d$, $e$, $h$ be real numbers, if the series converges, we have
\begin{align*}
&\int ^{1}_{0}\int ^{1}_{0}-\frac{\cos\left( \log\left( x^{a}y^{b}\right) \right) }{\left( c-x^{d}y^{e}\right) \log\left( x^{h}y^{g}\right) } dxdy\\
=&
\sum \limits^{\infty }_{n=0}\frac{1}{c^{n+1}} \Re\left( \int ^{\infty }_{0}\frac{dx}{\left( hx+dn+1+ai\right) \left( gx+en+1+bi\right) } \right), 
\end{align*}
\begin{align*}
&\int ^{1}_{0}\int ^{1}_{0}-\frac{\sin\left( \log\left( x^{a}y^{b}\right) \right) }{\left( c-x^{d}y^{e}\right) \log\left( x^{h}y^{g}\right) } dxdy\\= &
\sum \limits^{\infty }_{n=0}\frac{1}{c^{n+1}} \Im\left( \int ^{\infty }_{0}\frac{dx}{\left( hx+dn+1+ai\right) \left( gx+en+1+bi\right) } \right).
\end{align*}
\end{corollary}

\begin{proof}
We use the geometrical series formula,
\begin{equation*}
\sum \limits^{\infty }_{n=0}\frac{1}{c^{n+1}} \int ^{1}_{0}\int ^{1}_{0}\frac{-\sin \left( \log\left( x^{a}y^{b}\right) \right) x^{dn}y^{en}}{\log\left( x^{h}y^{g}\right) } dxdy=\int^{1}_{0}\int^{1}_{0}\frac{-\sin \left( \log\left( x^{a}y^{b}\right) \right)}{\left( c-x^{d}y^{e}\right) \log\left( x^{h}y^{g}\right) } dxdy.
\end{equation*}
and the theorem \ref{theor1}.
\end{proof}

\begin{theorem}\label{theor4}
Let $a, b, c$ be real numbers, if the integral converges,
\begin{equation*}
\int ^{1}_{0}\frac{x^{b-1}-x^{c-1}}{\left( 1+x^{a}\right) \log\left( x\right) } dx=
\log\left( \prod \limits^{\infty }_{n=0}\left( \frac{an+b}{an+c} \right) ^{(-1)^{n}}\right).
\end{equation*}
\end{theorem}
\begin{proof}
Just use the last corollary to get the double integral that belongs to the infinite series and make the substitution $y=u/x$ and $x=x$. 
\end{proof}

\begin{theorem}\label{theor5}
If the integrals are convergent, we have
\begin{equation*}
\int ^{1}_{0}\frac{x^{\frac{bc+c-d}{d} }\left( 1-x^{\frac{ad-bc-c+d}{ d} }\right) }{-\left( ad-bc-c+d\right) \log\left( x\right) } dx=\int ^{\infty }_{0}\frac{1}{\left( cx+a+1\right) \left( dx+b+1\right) } dx,
\end{equation*}
\begin{equation*}
\int ^{\infty }_{0}\frac{\sin\left( bx\right) -\sin\left( ax\right) }{\left( b-a\right) x} dx=\Re\left( \int ^{\infty }_{0}\frac{1}{\left( x+ai\right) \left( x+bi\right) } dx\right),
\end{equation*}
\begin{equation*}
\int ^{\infty }_{0}\frac{\cos\left( bx\right) -\cos\left( ax\right) }{\left( a-b\right) x} dx=\Im\left( \int ^{\infty }_{0}\frac{1}{\left( x+ai\right) \left( x+bi\right) } dx\right),
\end{equation*}
\end{theorem}

\begin{proof}
We get the double integral that corresponds to the single integral of the right side and then we make the change of variable $(x,y) \to (x,u/x)$.
\end{proof}

\section{Aplications to the Number Theory} 

The two following theorems are related to Lerch Transcendent and Riemann zeta function \cite{S}.

\begin{theorem}\label{NT}
Let $s>1 $ be an integer and $z \in \mathbb{C}$, then if the integral converges
\begin{equation*}
\Phi \left( z,s,a\right) =\int ^{1}_{0} \overset{s+1}{\dots} \int ^{1}_{0}\frac{-s\prod \limits^{s+1}_{n=1}x^{a-1}_{n}}{\left( 1-z\prod \limits^{s+1}_{n=1}x_{n}\right) \log\left( \prod \limits^{s+1}_{n=1}x_{n}\right) } \prod \limits^{s+1}_{n=1}dx_{n}
\end{equation*}
\begin{equation*}
\zeta \left( s \right) =\int ^{1}_{0} \overset{s+1}{\dots} \int ^{1}_{0}\frac{-s}{\left( 1-\prod \limits^{s+1}_{k=1}x_{k}\right) \log\left( \prod \limits^{s+1}_{k=1}x_{k}\right) } \prod \limits^{s+1}_{k=1}dx_{k}.
\end{equation*}
\end{theorem}

\begin{proof}
Take the following identities
\begin{equation*}
\Phi \left( z,s,a\right) =s\sum \limits^{\infty }_{k=0}\int ^{\infty }_{0}\frac{z^{k}}{\left( x+k+a\right) ^{s+1}} dx
\end{equation*}
\begin{equation*}
\frac{s}{s-1}\sum \limits^{\infty }_{n=1}\frac{1}{n^{s-1}} =\sum \limits^{\infty }_{n=1}\int ^{\infty }_{0}\frac{1}{\left( \frac{1}{s} x+n\right) ^{s}} dx.
\end{equation*}
Then we use the theorem \ref{theor1} and we interchange the limit and the sum thanks to the monotone convergence theorem.
\end{proof}

\begin{theorem} 
Let $s>1 $ be an integer and $z \in \mathbb{C}$, then if the integral converges
\begin{equation*}
\Phi \left( z,s,a\right) =\frac{-s}{\left( s-1\right) !} \int ^{1}_{0}\int ^{1}_{0}\frac{\left( xy\right) ^{a-1}\left( -\log\left( y\right) \right) ^{s-1}}{\left( 1-zxy\right) \log\left( xy\right) } dxdy
\end{equation*}
\begin{equation*}
\zeta \left( s\right) =\frac{s}{\left( s-1\right) !} \left( -1\right) ^{s}\int ^{1}_{0}\int ^{1}_{0}\frac{\log^{s-1}(y)}{\left( 1-xy\right) \log\left( xy\right) } dxdy.
\end{equation*}
\end{theorem}

\begin{proof}
We use the variable change $(x,y) \to (x,u/x)$ to the formula of the theorem \ref{NT} with $s=2$.
\end{proof}

\begin{theorem}\label{theor6}
Let $z \in \mathbb{C}$ if the integral converges
\begin{equation}
\Phi \left( z,s,a\right) \Gamma(s) =\int ^{1}_{0}\int ^{1}_{0}\frac{-s\left( xy\right) ^{a-1}\left( -\log\left( y\right) \right) ^{s-1}}{\left( 1-zxy\right) \log\left( xy\right) } dxdy
\end{equation}
\begin{equation*}
\zeta \left( s\right) \Gamma(s) =\int ^{1}_{0}\int ^{1}_{0}\frac{-s(-\log(y))^{s-1}}{\left( 1-xy\right) \log\left( xy\right) } dxdy.
\end{equation*}
\end{theorem}

\begin{theorem}\label{theor7}
Let $z \in \mathbb{C}$ if the integral converges
\begin{align*}
&\frac{\left( s-2\right) \left( -1\right)^{s} }{s} \Phi \left( z,s,a\right) \Gamma \left( s\right) \\
=&\int ^{1}_{0}\int ^{1}_{0}\frac{\left( xy\right) ^{a-1}}{\left( 1-zxy\right) \log\left( xy\right) } \sum \limits^{s-2}_{k=1}\dbinom{s-1}{k} \log^{k}\left( x\right) \log^{s-k-1}\left( y\right) dxdy
\end{align*}
\end{theorem}
\begin{proof}
Take Guillera's and Sondow's formula \cite{S2}
\begin{equation*}
\Phi \left( z,s,a\right) \Gamma \left( s\right) =\int ^{1}_{0}\int ^{1}_{0}\frac{\left( xy\right) ^{a-1}}{1-zxy} \left( -\log\left( xy\right) \right) ^{s-2}dxdy,
\end{equation*}
and use the binomial theorem for $\left( -\log\left( xy\right) \right) ^{s-2}$ and the formula $(2)$.
\end{proof}

The following theorems are about the Euler gamma constant.
In the following corollary we recover a formula by Sondow \cite{S1} and we give another proof of his formula.
\begin{corollary}
The following formula for $\gamma$ holds: 
\begin{equation*}
\gamma =-\int ^{1}_{0}\int ^{1}_{0}\frac{ 1-x}{(1-xy) \log (xy)} dxdy.
\end{equation*}
\end{corollary}
\begin{proof}

\begin{equation*}
-\int ^{1}_{0}\int ^{1}_{0}\frac{ 1-x}{(1-xy)\log(xy)} dxdy=\sum \limits^{\infty }_{n=0}\int ^{1}_{0}\int ^{1}_{0}(1-x)\frac{\left( xy\right) ^{n}}{-\log(xy)} dxdy 
\end{equation*}
\begin{equation*}
=\sum \limits^{\infty }_{n=0}\int ^{1}_{0}\int ^{1}_{0}\left( \frac{(xy)^{n}}{-\log (xy)} +\frac{x^{n+1}y^{n}}{\log (xy)} \right) dxdy = \sum \limits^{\infty }_{n=0}\left( \frac{1}{n+1} -\log \left( \frac{n+2}{n+1} \right) \right)=\gamma,
\end{equation*}
where the penultimate step follows from Theorem \ref{theor1}, and the last one from the definition of the Euler constant.
\end{proof}

We can generalized the last formula with the next theorem,
\begin{theorem}\label{theor8}
Let $ a, \, b, \, c$ and $ d $ be real numbers, if the integral converges
\begin{align*}
\int ^{1}_{0}\int ^{1}_{0}\frac{(xy)^{b}\left[ \left( xy\right) ^{d-1}-x^{a-1}y^{c-1}\right] }{-\left[ 1-(xy)^{b}\right] \log (xy)}dxdy &= \sum \limits^{\infty }_{n=1}\left( \frac{1}{bn+d} -\frac{\log\left( 1+\frac{c-a }{bn+a} \right) }{c-a } \right) 
\\
&= \frac{\log\left( \frac{\Gamma \left( \frac{b+c }{b} \right) }{\Gamma \left( \frac{a+b}{b} \right) } \right) }{c-a} -\frac{\psi \left( \frac{b+d}{b} \right) }{b},
\end{align*}
where $\psi$ is the the digamma function.
\end{theorem}

\begin{proof}
First we have these two integrals with $ a, \, b, \, c, \, d, \, k, \, g, \, h, \, i$ and $j$ being real numbers (with the condition that the integral converges)
\begin{equation*}
\int ^{\infty }_{0}\frac{1}{(ax+b+cn)(dx+g+kn)} dx=\frac{\log\left( \frac{d}{a} \right) -\log\left( \frac{kn+g}{cn+b} \right) }{-a (kn+g)+bd+cdn},
\end{equation*}
\begin{equation*}
\int ^{\infty }_{0}\frac{1}{(hx+i+jn)^{2}} dx=\frac{1}{h (jn+i)},
\end{equation*}
so, with the help of the theorem \ref{theor1}, and choosing $a=d=1$, $ k=c $ (with $ b\neq g$ and $c, k \neq 0$) and keeping the last hypothesis, we get
\begin{equation*}
\sum \limits^{\infty }_{n=1}\left( \frac{1}{h(jn+i)}+\frac{\log\left(\frac{cn+g}{cn+b} \right) }{b-g} \right),
\end{equation*}
\begin{equation*}
=\sum \limits^{\infty }_{n=1}\left( \int ^{1}_{0}\int ^{1}_{0}\frac{(xy)^{i+jn-1}}{-h\log (xy)} dxdy-\int ^{1}_{0}\int ^{1}_{0}\frac{x^{b+cn-1}y^{g+cn-1}}{-\log (xy)} dxdy\right),
\end{equation*}
Then, doing $a=d=h=1$, $c=j=k$ (so $c, j, k\neq 0$), and keeping the condition that the series converges, we see that the sum above is equal to
\begin{equation*}
\sum \limits^{\infty }_{n=1}\int ^{1}_{0}\int ^{1}_{0}\frac{(xy)^{i-1+jn}-x^{b-1}y^{g-1}(xy)^{cn}}{-\log(xy)} dydx,
\end{equation*}
so
\begin{equation*}
\sum \limits^{\infty }_{n=1}\left( \frac{1}{cn+i} +\frac{\log\left( 1+\frac{g-b}{cn+b} \right) }{b-g} \right) 
=\sum \limits^{\infty }_{n=1}\int ^{1}_{0}\int ^{1}_{0}\frac{(xy)^{cn}\left[ (xy)^{i-1}-x^{b-1}y^{g-1}\right] }{-\log (xy)} dxdy.
\end{equation*}
According to the monotone convergence theorem we can interchange the sum and the double integral. Hence

\begin{equation*}
\int ^{1}_{0}\int ^{1}_{0}\frac{(xy)^{c}\left[ (xy)^{i-1}-x^{b-1}y^{g-1}\right] }{-(1-(xy)^{c})\log(xy)} dxdy=\sum \limits^{\infty }_{n=1}\left( \frac{1}{cn+i} -\frac{\log\left(1+\frac{g-b}{cn+b} \right) }{g-b} \right). 
\end{equation*}
Finally, rename the parameters.
\end{proof}

Finally, we give a double integral for Digamma function.
\begin{theorem}\label{theor9}
Let $s\in\mathbb{C}$, if the integral converges
\begin{equation*}
\psi \left( s\right) =\int ^{1}_{0}\int ^{1}_{0}\frac{\left( xy\right) ^{s-1}-y}{\left( 1-xy\right) \log\left( xy\right) } dxdy\end{equation*}
where $ \psi \left( s\right) $ is the Digamma function.
\end{theorem}
\begin{proof}
Using the sum of the geometric series we get
\begin{equation*}
\int ^{1}_{0}\int ^{1}_{0}\frac{\left( xy\right) ^{s-1}-y}{\left( 1-xy\right) \log\left( xy\right) } dxdy=\int ^{1}_{0}\int ^{1}_{0}\sum \limits^{\infty }_{n=0}\left( xy\right) ^{n}\frac{\left( xy\right) ^{s-1}-y}{\log\left( xy\right) } dxdy,\end{equation*}
Finally, we solve the integral with the help of the main theorem
\begin{equation*}
\sum \limits^{\infty }_{n=0}\int ^{1}_{0}\int ^{1}_{0}\left( \frac{\left( xy\right) ^{n+s-1}}{\log\left( xy\right) } -\frac{x^{n}y^{n+1}}{\log\left( xy\right) } \right) dxdy=\sum \limits^{\infty }_{n=0}\left( \log\left( 1+\frac{1}{n+1} \right) -\frac{1}{n+s} \right) =\psi \left( s\right).\end{equation*}
\end{proof}

\section{Some Ramanujan formulas}
The Ramanujan's paper entitled \textit{Some Definite Integrals} which appeared in Messenger of Mathematics (1915), \cite{R} included, among others, the following formulae: 
\begin{equation*}
\int^{\infty }_{0}\frac{dx}{\left( 1+x^{2}\right) \left( 1+r^{2}x^{2}\right) \left( 1+r^{4}x^{3}\right) \cdots } =
\frac{\pi }{2\left( 1+r+r^{3}+r^{6}+r^{10}+\cdots \right) },
\end{equation*}
\begin{equation*}
\int ^{\infty }_{0}\frac{1}{\left( 1+\frac{x^{2}}{a^{2}} \right) \left( 1+\frac{x^{2}}{\left( a+1\right) ^{2}} \right) \left( 1+\frac{x^{2}}{\left( a+2\right) ^{2}} \right) \cdots } dx=\frac{1}{2} \sqrt{\pi } \frac{\Gamma \left( a+\frac{1}{2} \right) }{\Gamma \left( a\right)},
\end{equation*}

Using our main  theorem we get another expressions for these formulae.

\begin{theorem}
If $ r $ is a real number such that $|r|<1$ and $k\in\mathbb{N}$, then if the integral converges

\begin{align*}&\int ^{\infty }_{0}\overset{2k}{\dots} \int ^{\infty }_{0}\frac{ \cos \left( \sum \limits^{2k}_{n=1} (-1)^{n+1}x_{n}\right) }{\sum \limits^{k}_{n=1} r^{n-1}x_{2n}+ \sum \limits^{k}_{n=1} r^{n-1}x_{2n-1} } \prod \limits^{2k}_{n=1}dx_{n}\\ &=\frac{\pi }{2\left( 1+r+r^{3}+r^{6}+r^{10}+\cdots+r^{\frac{1}{2}k(k-1)} \right) } ,\end{align*}
\begin{equation*}\int ^{\infty }_{0}\overset{2k}{\dots} \int ^{\infty }_{0}\frac{\sin \left( \sum \limits^{2k}_{n=1}(-1)^{n+1} x_{n}\right) }{ \sum \limits^{k}_{n=1} r^{n-1}x_{2n}+ \sum \limits^{k}_{n=1} r^{n-1}x_{2n-1} } \prod \limits^{2k}_{n=1}dx_{n}=0.\end{equation*}
\end{theorem}

\begin{theorem}
Let $ a $ be a positive real numbers and $k\in\mathbb{N}$, then if the integral converges
\begin{equation*} \lim \limits_{k\to \infty } \int ^{\infty }_{0}\overset{2k}{\dots} \int ^{\infty }_{0}\frac{\cos \left( \sum \limits^{2k}_{n=1}(-1)^{n+1}x_{n}\right) }{\sum \limits^{k}_{n=1} \frac{x_{2n}}{a+n-1}+\sum \limits^{k}_{n=1} \frac{x_{2n-1}}{a+n-1}} \prod \limits^{2k}_{n=1}dx_{n}= \frac{1}{2} \sqrt{\pi } \frac{\Gamma \left( a+\frac{1}{2} \right) }{\Gamma \left( a\right)} \end{equation*}
\begin{equation*}\int ^{\infty }_{0}\overset{2k}{\dots} \int ^{\infty }_{0}\frac{\sin \left( \sum \limits^{2k}_{n=1}(-1)^{n+1}x_{n}\right) }{ \sum \limits^{k}_{n=1} \frac{x_{2n}}{a+n-1}+\sum \limits^{k}_{n=1} \frac{x_{2n-1}}{a+n-1} } \prod \limits^{2k}_{n=1}dx_{n}=0.\end{equation*}
\end{theorem}




\section{Examples}
In this section we'll see some examples and applications of all we have seen before.

\begin{example}
(Theorem \ref{theor1})
Taking $k=2$, setting $x_1=x$ and $x_2=y$, and choosing $a_1=1, b_1=1, c_1=1, a_2=2, b_2=1$ and $c_2=1$ values of the parameters, we see 
\begin{equation*}
\int ^{\infty }_{0}\frac{1}{(x+1+i)(2x+1+i)}dx=\left(\frac{1}{2} -\frac{i}{2}\right)\log(2),
\end{equation*}
we get
\begin{equation*}
\int ^{\infty }_{0}\int ^{\infty }_{0}\frac{e^{-x-y}\sin (x+y)}{2x+y} dydx=
\int ^{\infty }_{0}\int ^{\infty }_{0}\frac{e^{-x-y}\cos (x+y)}{2x+y} dydx=\frac{\log(2)}{2}. 
\end{equation*}
\end{example}

\begin{example} (Theorem \ref{theor1})
Taking $k=3$, setting $x_1=x$, $x_2=y$ and $ x_3=z $, and choosing $a_1=1, b_1=1, c_1=1, a_2=2, b_2=1, c_2=1, a_3=1, b_3=2$ and $c_3=1$ values of the parameters, we see 
\begin{equation*}
I= \int ^{\infty }_{0}\frac{1}{(x+1+i)(2x+1+i)(x+2+i)} dx,
\end{equation*}
we obtain
\begin{equation*}
\int ^{\infty }_{0}\int ^{\infty }_{0}\int ^{\infty }_{0}\frac{e^{-x-y-2z}\cos(x+y+z)}{x+2y+z} dxdy=
\frac{1}{40}\left(\pi-2\arctan\left(\frac{4}{3}\right)+4\hbox{arctanh}\left(\frac{3}{253}\right) \right), 
\end{equation*}
and
\begin{equation*}
\int ^{\infty }_{0}\int ^{\infty }_{0}\int ^{\infty }_{0}\frac{e^{-x-y-2z}\sin(x+y+z)}{x+2y+z} dxdy=-
\frac{1}{40}\left(3\pi+\log(25)-6\left(\log(8)+\arctan\left(\frac{4}{3}\right)\right)\right), 
\end{equation*}
\end{example}

\begin{example}(Theorem \ref{theor1})
Taking $k=3$, setting $x_1=x$, $x_2=y$ and $ x_3=z $, and choosing $a_1=1, b_1=4, c_1=0, a_2=2, b_2=3, c_2=0, a_3=3, b_3=2$ and $c_3=0$ values of the parameters, we see 
\begin{align*}
\int ^{1}_{0}\int ^{1}_{0}\int ^{1}_{0}\frac{x^{3}y^{2}z}{-\log\left( xy^{2}z^{3}\right) } dxdydz&=\int ^{\infty }_{0}\int ^{\infty }_{0}\int ^{\infty }_{0}\frac{e^{-4x-3y-2z}}{x+2y+3z} dxdydz \\
&=\int ^{\infty }_{0}\frac{1}{\left( x+4\right) \left( 2x+3\right) \left( 3x+2\right) } dx \\
&=\frac{1}{50} \log\left( \frac{2187}{512} \right).
\end{align*}
\end{example}

\begin{example}(Theorem \ref{theor1})
Choosing $a_n=1, b_n=1$ and $c_n=0$ values of the parameters, we see
\begin{equation*}
\int ^{\infty }_{0}\overset{n}{\dots} \int ^{\infty }_{0}\frac{e^{-\sum \limits^{n}_{i=1}x_{i}}}{\sum \limits^{n}_{i=1}x_{i}} \cos\left(\sum \limits^{n}_{i=1}x_{i}\right) \prod \limits^{n}_{i=1}dx_{i}=\frac{2^{\frac{1-n}{2} }}{n-1} \cos\left( \frac{\pi }{4} \left( n-1\right) \right)
\end{equation*}
\begin{equation*}
\int ^{\infty }_{0}\overset{n}{\dots} \int ^{\infty }_{0}\frac{e^{-\sum \limits^{n}_{i=1}x_{i}}}{\sum \limits^{n}_{i=1}x_{i}} \sin\left(\sum \limits^{n}_{i=1}x_{i}\right) \prod \limits^{n}_{i=1}dx_{i}=\frac{2^{\frac{1-n}{2} }}{n-1} \sin\left( \frac{\pi }{4} \left( n-1\right) \right)
\end{equation*}
If $k=2015$,
\begin{equation*}
\int ^{\infty }_{0}\overset{2015}{\dots} \int ^{\infty }_{0}\frac{e^{-\sum \limits^{2015}_{i=1}x_{i}}}{\sum \limits^{2015}_{i=1}x_{i}} \sin\left(\sum \limits^{2015}_{i=1}x_{i}\right) \prod \limits^{2015}_{i=1}dx_{i}=-\frac{2^{-1007}}{2014} 
\end{equation*}
\end{example}

\begin{example}(Theorem \ref{theor1})
Taking $k=n$, and choosing $a_n=1, b_n=1$ and $c_n=0$ values of the parameters, we see
\begin{equation*}
\int ^{1}_{0}\overset{n}{\dots}\int ^{1}_{0}\frac{-1}{\log\left( \prod \limits^{n}_{i=1}x_{i}\right) } \prod \limits^{n}_{i=1}dx_{i}= \int ^{\infty }_{0}\frac{1}{\left( x+1\right) ^{n}} dx= \frac{1}{n-1} \end{equation*}
\end{example}
\begin{example}$($Theorem \ref{theor2}$)$
If $j=k$ and $k=n$,
\begin{equation*}\int ^{1}_{0}\overset{n}{\dots}\int ^{1}_{0}\frac{-1}{\log^{k}\left( \prod \limits^{n}_{i=1}x_{i}\right) } \prod \limits^{n}_{i=1}dx_{i}=\frac{1}{\left( n-1\right) \left( n-2\right) \cdots \left( n-k\right) }. \end{equation*}
If $j=n-1$ and $k=n$,
\begin{equation}
\int ^{1}_{0}\overset{n}{\dots}\int ^{1}_{0}\frac{-1}{\log^{n-1}\left( \prod \limits^{n}_{i=1}x_{i}\right) } \prod \limits^{n}_{i=1}dx_{i}=\frac{1}{\left( n-1\right) !}. \end{equation}
If $n \to \infty$,
\begin{equation*}
\lim \limits_{n\to \infty }\int ^{1}_{0}\overset{n}{\dots}\int ^{1}_{0}\frac{-1}{\log^{k}\left( \prod \limits^{n}_{i=1}x_{i}\right) } \prod \limits^{n}_{i=1}dx_{i}=0 \qquad \forall k\geq 1.\end{equation*}
And if $j=2$ and $k=1500$,
\begin{equation*}
\int ^{1}_{0}\overset{1500}{\dots}\int ^{1}_{0}\frac{-1}{\log^{2}\left( \prod \limits^{1500}_{i=1}x_{i}\right) } \prod \limits^{1500}_{i=1}dx_{i}=\frac{1}{2245502}.\end{equation*}
\end{example}

\begin{example}$($Theorem \ref{theor99}$)$
If $m=1$, $a_{1}=a_{2}=1$, $b_{i}=0$, $c_{i}=1$, $j=n$ and $k=2$,
\begin{equation*}
\int _{0}^{\infty}\overset{n}{\dots}\int _{0}^{\infty}\frac{e^{-\sum x_{i}}}{\left(1+\sum x_{i}\right)^{2}}\prod \limits^{n}_{i=1}dx_{i}=\int_{0}^{\infty}\int_{0}^{\infty}\frac{e^{-x-y}}{(x+y+1)^n}dxdy=\frac{-1+enE_{n-1}(1)}{n-1}.
\end{equation*}
where $E_{n}(x)$ is the Exponential integral $E$.
\begin{equation*}
\end{equation*}
\end{example}

\begin{example}(Corollary \ref{theor3})
If $k=2$ and doing a variable change we can assure
\begin{align*}
\int ^{1}_{0}\int ^{1}_{0}\frac{1}{\left( 1-xy\right) \left( 1-\log\left( xy\right) \right) } dxdy &=\int ^{1}_{0}\frac{-\log\left( x\right)}{\left( 1-x\right) \left( 1-\log\left( x\right) \right) } dx \\ &=\sum \limits^{\infty }_{n=1}\left( \frac{1}{n} -e^{n}\int ^{\infty }_{n}\frac{e^{-x}}{x} dx\right) \\ &=\gamma+\sum \limits^{\infty }_{n=1}\left( \log\left(1+\frac{1}{n}\right) -e^{n}\int ^{\infty }_{n}\frac{e^{-x}}{x} dx\right),\end{align*}
\begin{align*}
\int ^{1}_{0}\frac{-\log\left( x\right) x^{\alpha -1}}{\left( 1-x^{\alpha }\right) \left( 1-\log\left( x\right) \right) } dx &=\sum \limits^{\infty }_{n=1}\left( \frac{1}{\alpha n} -e^{\alpha n}\int ^{\infty }_{\alpha n}\frac{e^{-x}}{x} dx\right) \\ &=\frac{1}{\alpha}\gamma+\sum \limits^{\infty }_{n=1}\left( \frac{1}{\alpha}\log\left(1+\frac{1}{n}\right) -e^{\alpha n}\int ^{\infty }_{\alpha n}\frac{e^{-x}}{x} dx\right).\end{align*}
with $\alpha \in \mathbb{R^{+}}$.
\end{example}

\begin{example}(Corollary \ref{cor1})
Letting $a=b=c=d=e=h=g=1$ in the first formula, we have
\begin{equation*}
\int ^{1}_{0}\int ^{1}_{0}\frac{-\sin\left( \log\left( xy\right) \right) }{\left( 1-xy\right) \log\left( xy\right) } dxdy= 
\sum \limits^{\infty }_{n=0}\frac{1}{\left( n+1\right) ^{2}+1} = \frac{1}{2} \left( \pi \coth\left( \pi \right) -1\right),
\end{equation*}
If we make the substitution $y=u/x$ and $x=x$ to the integral, we get
\begin{equation*}
\int ^{\infty }_{0}\frac{\sin\left( x\right) }{1-e^{x}} dx=\frac{1}{2} \left( -\pi \coth\left( \pi \right) +1\right). \end{equation*}
\end{example}

\begin{example}(Theorem \ref{theor4})
If $a=3$, $b=3$ and $c=1$:
\begin{equation*}
\int^{1}_{0} \frac{x^{2}-1}{\left( 1+x^{3}\right) \log\left( x\right) } dx=\log\left( \prod \limits^{\infty }_{n=0}\left( \frac{3n+3}{3n+1} \right)^{(-1)^{n}}\right) =\log\left( \frac{\Gamma \left( \frac{1}{6} \right) }{\sqrt{\pi } \Gamma \left( \frac{2}{3} \right) } \right).
\end{equation*}
\end{example}

\begin{example} $($Theorem \ref{theor5}$)$
If we change $a=-1$ and $b=1$ in the 2nd formula we prove one of the Dirichlet's formulas
\begin{equation*}
\int ^{\infty }_{0}\frac{\sin\left( x\right) }{x} dx=\frac{\pi }{2}. 
\end{equation*}
\end{example}

\begin{example}(Theorem \ref{theor6})
If $s=3$
\begin{equation*}
\zeta \left( 3\right) =\frac{-3}{2}\int ^{1}_{0}\int ^{1}_{0}\frac{\log^{2}y}{\left( 1-xy\right) \log\left( xy\right) } dxdy.
\end{equation*}
\end{example}

\begin{example}(Theorem \ref{theor6})
We know $\Phi \left( -1,2,1/2\right) =4G$, so if $z=-1$, $s=2$ and $a=1/2$ 
\begin{equation*}G=\frac{1}{2} \int ^{1}_{0}\int ^{1}_{0}\int ^{1}_{0}\frac{-1}{\sqrt{xyz} \left( 1+xyz\right) \log\left( xyz\right) } dxdydz=\int ^{1}_{0}\int ^{1}_{0}\frac{ \log\left( y\right) }{2\sqrt{xy} \left( 1+xy\right) \log\left( xy\right) } dxdy.\end{equation*}
where $G$ is the Catalan's constant.
\end{example}

\begin{example}(Theorem \ref{theor6})
If $z=-1$,$s=1$ and $a=1/4$
\begin{equation*}
\int ^{1}_{0}\int ^{1}_{0}\frac{-1}{\sqrt[4]{\left( xy\right) ^{3}} \left( 1+xy\right) \log\left( xy\right) } dxdy=\frac{\pi +2\coth^{-1}\left( \sqrt{2} \right) }{\sqrt{2} } \end{equation*}
\end{example}

\begin{example}(Theorem \ref{theor7})
If $z=1$, $s=7$ and $a=1$,
\begin{equation*}
\zeta \left( 7\right) =\frac{-7}{3600} \int ^{1}_{0}\int ^{1}_{0}\frac{\log\left( x\right) \log\left( y\right) }{\left( 1-x y\right) \log\left( xy\right) } \left[ 12 \log^4(y)+30\log(x)\log^3(y) +20\log^2(x)\log^2(y)\right]
\end{equation*}
\end{example}

\begin{example}(Theorem \ref{theor8})
If $a=2, b=2, c=1$ and $d=0$, 
\begin{equation*}
\int ^{1}_{0}\int ^{1}_{0}\frac{(x y)^{2}\left[ \left( xy\right) ^{-1}-x\right] }{-\left[ 1-(xy)^{2}\right] \log (xy)}dxdy = -\log\left(\frac{ \sqrt{\pi}}{2}\right)+\frac{\gamma}{2},
\end{equation*}
If $a=4, b=5, c=8$ and $d=10$, 
\begin{equation*}
\int ^{1}_{0}\int ^{1}_{0}\frac{(xy)^{5}\left[ \left( xy\right) ^{9}-x^{3}y^{7}\right] }{-\left[ 1-(xy)^{5}\right] \log (xy)}dxdy = \frac{1}{4}\log\left( \frac{\Gamma \left(\frac{13}{5}\right)}{\Gamma \left(\frac{9}{5}\right)}\right)+\frac{1}{5}\gamma-\frac{3}{10}.
\end{equation*}
\end{example}

\begin{example}(Teorema \ref{theor9})
Thanks to the double integral of Digamma function we can get
\begin{equation*}
\sum \limits^{\infty }_{n=1}\frac{\psi \left( n\right) }{n\left( n+1\right) } =1-\gamma \end{equation*}
\begin{equation*}
\sum \limits^{\infty }_{n=1}\frac{\psi \left( \frac{n}{2} \right) }{n\left( n+1\right) } =2-\gamma -\frac{\pi ^{2}}{6} -\log^{2}\left( 2\right) \end{equation*}
\begin{equation*}
\sum \limits^{\infty }_{n=1}\frac{\psi \left( \frac{n}{4} \right) }{n\left( n+1\right) } =4-2G-\gamma -\frac{19\pi ^{2}}{48} +\frac{\pi }{4} \log\left( 2\right) -\frac{5}{4} \log^{2}\left( 2\right) \end{equation*}
\begin{equation*}
\sum \limits^{\infty }_{n=1}\left( -1\right) ^{n}\frac{\psi \left( n\right) }{n} =\frac{1}{2} \log\left( 2\right) \left( 2\gamma +\log\left( 2\right) \right) \end{equation*}
\begin{equation*}
\sum \limits^{\infty }_{n=1}\left( -1\right) ^{n}\frac{\psi \left( \frac{n}{2} \right) }{n} =\frac{\pi ^{2}}{12} +\log\left( 2\right) \left( \gamma +\log\left( 2\right) \right) \end{equation*}
\begin{equation*}
\sum \limits^{\infty }_{n=1}\left( -1\right) ^{n}\frac{\psi \left( \frac{n}{4} \right) }{n} =\frac{11\pi ^{2}}{48} +\gamma \log\left( 2\right) +\frac{7}{4} \log^{2}\left( 2\right) \end{equation*}
\begin{equation*}
\sum \limits^{\infty }_{n=1}\frac{\psi \left( n\right) }{n^{2}} =-\frac{\pi ^{2}}{6} \gamma +\zeta \left( 3\right) \end{equation*}
\begin{equation*}
\sum \limits^{\infty }_{n=1}\frac{\psi \left( n\right) }{n\left( n+2\right) } =\frac{7}{8} -\frac{3}{4} \gamma \end{equation*}
\begin{equation*}
\sum \limits^{\infty }_{n=1}\frac{\psi \left( \frac{n}{2} \right) }{n\left( n+2\right) } =\frac{3}{4} -\frac{3}{4} \gamma -\log\left( 2\right) \end{equation*}
\begin{equation*}
\sum \limits^{\infty }_{n=1}\frac{\psi \left( \frac{n}{4} \right) }{n\left( n+2\right) } =\frac{1}{48} \left( -36\gamma +\left( 12-11\pi \right) \pi -12\left( -6+\log^{2}\left( 2\right) +6\log\left( 2\right) \right) \right) .\end{equation*}
where $G$ is the Catalan's constant.
\end{example}

\section*{Acknowledgment}
Special thanks to Jes\'us Guillera and Jonathan Sondow for their help in the editing of the article. The first people who rated my work. I also acknowledge to the anonymous referee whose comments have been very helpful in improving the presentation.


\end{document}